\documentclass[11pt,reqno]{amsart}
\usepackage[utf8]{inputenc}
\usepackage{amscd,amssymb,amsmath,amsthm}
\usepackage[arrow,matrix]{xy}
\usepackage{graphicx}
\usepackage{epstopdf}
\usepackage{color}
\usepackage{multicol}
\newtheorem{thm}{Theorem}
\newtheorem{defn}{Definition}
\newtheorem{lemma}{Lemma}
\newtheorem{pro}{Proposition}
\newtheorem{rk}{Remark}

\newtheorem{ex}{Example}

\numberwithin{equation}{section} 

\begin{document}
\title [A DISCRETE-TIME DYNAMICAL SYSTEM OF MOSQUITO POPULATION]
{A discrete-time dynamical system of wild mosquito population with Allee effects}

\author{U.A. Rozikov, Z.S. Boxonov}

\address{U.\ A.\ Rozikov \begin{itemize}
 \item[] V.I.Romanovskiy Institute of Mathematics of Uzbek Academy of Sciences;
\item[] AKFA University, 1st Deadlock 10, Kukcha Darvoza, 100095, Tashkent, Uzbekistan;
\item[] Faculty of Mathematics, National University of Uzbekistan.
\end{itemize}} \email {rozikovu@yandex.ru}

\address{Z.\ S.\ Boxonov\\ V.I.Romanovskiy Institute of Mathematics of Uzbek Academy of Sciences,
Tashkent, Uzbekistan.}
 \email {z.boxonov@mathinst.uz}

\begin{abstract} We study a discrete-time dynamical system of wild mosquito population with parameters: $\beta$ - the birth rate of adults;
$\alpha$ - maximum emergence rate; $\mu>0$ - the death rate of adults;  $\gamma$ - Allee effects. We prove that if $\gamma\geq\frac{\alpha(\beta-\mu)}{\mu^2}$
then the mosquito population dies and if $\gamma<\frac{\alpha(\beta-\mu)}{\mu^2}$ holds then extinction or survival of the mosquito population depends on their initial state.
\end{abstract}


 \maketitle

\section{Introduction}

Mosquito population control is a vital public-health practice throughout the world and especially in the tropics because mosquitoes spread many diseases, such as malaria and various viruses.

Today several kind of mathematical models of mosquito population are known (see \cite{B}, \cite{J}, \cite{RV}, \cite{Rpd}  and references therein).

A mathematical model of mosquito dispersal in continuous time was investigated in \cite{LP}.
Recently, in \cite{RV} a discrete-time dynamical system, generated by an evolution operator of this mosquito population is studied.

It is known that during a lifetime  mosquitoes undergo complete metamorphosis going through four distinct stages of development: egg, larva, pupa and adult  \cite{B}, \cite{Mosquito}.

In \cite{J.Li} continuous-time model of mosquito population with Allee effects  is studied, here
 we consider a discrete-time dynamical system of this model. Consider a wild mosquito population without the presence of sterile mosquitoes and in a simplified stage-structured population, one groups the three aquatic stages into the larvae class by $x$, and divide the mosquito population into the larvae class and the adults, denoted by $y$. Moreover, assume that the density dependence exists only in the larvae stage \cite{J.Li}.

Denote the birth rate, i.e., the oviposition rate of adults by $\beta(t)$; the rate of emergence from larvae to adults by a function of the larvae with the form of $\alpha(1-k(x))$, where $\alpha>0$ is the maximum emergence rate, $0\leq k(x)\leq 1$, with $k(0)=0, k'(x)>0$, and $\lim\limits_{x\rightarrow
\infty}k(x)=1$, is the functional response due to
the intraspecific competition \cite{J}.
Moreover, we assume the death rate of larvae be a linear function, denoted by $d_{0}+d_{1}x$, and
the death rate of adults be constant, denoted by $\mu$. Then, in the absence of sterile mosquitoes, we get the following
system of equations:
\begin{equation}\label{sys}
\left\{%
\begin{array}{ll}
    \frac{dx}{dt}=\beta(t) y-\alpha(1-k(x))x-(d_{0}+d_{1}x)x,\\[3mm]
    \frac{dy}{dt}=\alpha(1-k(x))x-\mu y. \\
\end{array}%
\right.\end{equation}
    We further assume a functional response for $k(x)$, as in
    \cite{J}, in the form $$k(x)=\frac{x}{1+x}.$$

In \cite{J}, \cite{J.Li} the dynamical system (\ref{sys}) was studied for $\beta(t)=\beta=const$ (i.e. when mosquito adults have no difficulty to find their mates such that no Allee effects are concerned) and the discrete-time version of this model was considered in \cite{BR1} and \cite{BR2}.
A component Allee effect is defined as a decrease in any component of fitness with decreasing population size or density. A decrease in the probability of a female mating with decreasing male density is therefore a component Allee effect, and is generally referred to as a "mate-finding Allee effect" \cite{XF}.

In the case where adult mosquitoes have difficulty in finding their mates, Allee effects are
included and the adult birth rate is given by $$\beta(t)=\frac{\beta y(t)}{\gamma+y(t)},$$
where $\gamma$ is the Allee effect constant.

Then stage-structured wild mosquito population model
is given by \cite{J.Li}:
\begin{equation}\label{system}
\left\{%
\begin{array}{ll}
    \frac{dx}{dt}=\frac{\beta y^2}{\gamma+y}-\frac{\alpha x}{1+x}-(d_{0}+d_{1}x)x,\\[3mm]
    \frac{dy}{dt}=\frac{\alpha x}{1+x}-\mu y. \\
\end{array}%
\right.\end{equation}

In this paper (as in \cite{JU} -\cite{MS}, \cite{MPR} -\cite{RS}) we study the discrete time dynamical systems associated to the system (\ref{system}). Define the operator
$W:{\mathbb{R}^2}\rightarrow {\mathbb{R}^2}$ by
\begin{equation}\label{systema}
\left\{%
\begin{array}{ll}
    x'=\frac{\beta y^2}{\gamma+y}-\frac{\alpha x}{1+x}-(d_{0}+d_{1}x)x+x,\\[3mm]
    y'=\frac{\alpha x}{1+x}-\mu y+y, \\
\end{array}%
\right.\end{equation}
where $\alpha >0, \beta >0, \gamma>0, \mu >0,  \ d_{0}\geq0,\ d_{1}\geq0.$

In this paper we consider the operator $W$ (defined by (\ref{systema})) for the case $d_{0}=d_{1}=0$ and our
aim is to study trajectories $z^{(n)}=W(z^{(n-1)})$, $n\geq 1$ of any initial point $z^{(0)}=(x^{(0)}, y^{(0)})$.

Note that this system (\ref{systema}), for the case when $d_{0}\ne 0$ or $d_{1}\ne 0$ is not studied yet.

\section{Dynamical system generated by the operator (\ref{systema})}

We assume
\begin{equation}\label{par}
d_{0}=d_{1}=0
\end{equation}
 then (\ref{systema}) has the following form
\begin{equation}\label{systemacase}
W_{0}:\left\{%
\begin{array}{ll}
    x'=\frac{\beta y^2}{\gamma+y}-\frac{\alpha x}{1+x}+x,\\[3mm]
    y'=\frac{\alpha x}{1+x}-\mu y+y. \\
\end{array}%
\right.\end{equation}

It is easy to see that if
\begin{equation}\label{parametr}
0<\alpha\leq1, \ \beta >0, \ \gamma>0,\  0<\mu\leq1
\end{equation}
then operator (\ref{systemacase}) maps $\mathbb{R}_{+}^{2}=\{(x,y)\in\mathbb{R}^{2}: x\geq0, y\geq0\}$ to itself.

\subsection{Fixed points.}\

A point $z\in\mathbb{R}_{+}^{2}$ is called a fixed point of $W_{0}$ if
$W_{0}(z)=z$.
%The set of fixed points is denoted by Fix$(W_{0})$.

For fixed point of $W_{0}$ the following holds.
\begin{pro}\label{fixed} The fixed points for (\ref{systemacase}) are as follows:
\begin{itemize}
  \item If $\beta\leq\mu(1+\frac{\gamma\mu}{\alpha})$ then the operator (\ref{systemacase}) has a unique fixed point $z=(0,0).$
  \item If $\beta>\mu(1+\frac{\gamma\mu}{\alpha})$ then mapping (\ref{systemacase}) has two
fixed points $$z_{1}=(0,0), \ \ z_{2}=\left(\frac{\gamma\mu^2}{\alpha(\beta-\mu)-\gamma\mu^2},\frac{\gamma\mu}{\beta-\mu}\right).$$
\end{itemize}
\end{pro}
\begin{proof} We need to solve
\begin{equation}\label{Fsystema}
\left\{%
\begin{array}{ll}
    x=\frac{\beta y^2}{\gamma+y}-\frac{\alpha x}{1+x}+x,\\[3mm]
    y=\frac{\alpha x}{1+x}-\mu y+y \\
\end{array}%
\right.\end{equation}
It is easy to see that $x_{1}=0,\ y_{1}=0$ and $x_{2}=\frac{\gamma\mu^2}{\alpha(\beta-\mu)-\gamma\mu^2},\ y_{2}=\frac{\gamma\mu}{\beta-\mu}$ are solution to (\ref{Fsystema}).
If $\beta\leq\mu(1+\frac{\gamma\mu}{\alpha})$ then $x_{2}\notin\mathbb{R}_{+},$ otherwise $x_{2}\in\mathbb{R}_{+}.$
\end{proof}

\subsection{The type of the fixed point.}\

Now we shall examine the type of the fixed point.
\begin{defn}\label{d1}
(see \cite{D}) A fixed point $z$ of an operator $W$ is called
\texttt{hyperbolic} if its Jacobian $J$ at $z$ has no eigenvalues on the
unit circle.
\end{defn}

\begin{defn}\label{d2}
(see \cite{D}) A hyperbolic fixed point $z$ is called:
\begin{itemize}
\item[1)] \texttt{attracting} if all the eigenvalues of the Jacobian matrix $J(z)$
are less than 1 in absolute value;

\item[2)] \texttt{repelling} if all the eigenvalues of the Jacobian matrix $J(z)$
are greater than 1 in absolute value;

\item[3)] a \texttt{saddle} otherwise.
\end{itemize}
\end{defn}
To find the type of a fixed point of the operator (\ref{systemacase})
we write the Jacobian matrix:

$$J(z)=J_{W_{0}}=\left(%
\begin{array}{cc}
  1-\frac{\alpha}{(1+x)^2} & \frac{\beta y(2\gamma+y)}{(\gamma+y)^2} \\
  \frac{\alpha}{(1+x)^2} & 1-\mu \\
\end{array}%
\right).$$

The eigenvalues of the Jacobian matrix at the fixed point $(0.0)$ are as follows
$$\lambda_{1}=1-\alpha,\ \lambda_{2}=1-\mu.$$
By (\ref{parametr}) we have $0\leq\lambda_{1,2}<1$.

Let
\begin{equation}\label{x^*y^*}x^*=\frac{\gamma\mu^2}{\alpha(\beta-\mu)-\gamma\mu^2},\ y^*=\frac{\gamma\mu}{\beta-\mu}.
\end{equation}
We calculate eigenvalues of Jacobian matrix at the fixed point ($x^*,y^*)$. If we denote $1-\lambda=\Lambda$ then we obtain
$$\Lambda^2-(\mu+A)\Lambda+A(\mu-B)=0,$$
where $A=\frac{\alpha}{(1+x^*)^2},\ B=\frac{\beta y^*(2\gamma+y^*)}{(\gamma+y^*)^2}.$
\begin{equation}\label{l12}
\Lambda_{1}=1-\lambda_{1}=\frac{1}{2}\left(\mu+A+\sqrt{(\mu-A)^2+4AB}\right), \ \Lambda_{2}=1-\lambda_{2}=\frac{1}{2}\left(\mu+A-\sqrt{(\mu-A)^2+4AB}\right).
\end{equation}

The inequality $|\lambda_{1,2}|<1$ is equivalent to $0<\Lambda_{1,2}<2.$

Since $\mu+A<2$  (see \ref{parametr}), the inequality $0<\Lambda_{1,2}<2$ is equivalent to the following:
$$\left\{%
\begin{array}{ll}
    \mu+A-\sqrt{(\mu-A)^2+4AB}>0  \\[2mm]
    \mu+A+\sqrt{(\mu-A)^2+4AB}<4.
\end{array}%
\right.$$
From the first inequality $\mu+A-\sqrt{(\mu-A)^2+4AB}>0$ of the system  we have $\mu>B$.

If we consider
$$B=\frac{\beta y^*(2\gamma+y^*)}{(\gamma+y^*)^2}=\beta(1-\frac{\gamma^2}{(\gamma+y^*)^2})=\beta-\frac{1}{\beta}(\beta-\mu)^2, \ \ \beta>\mu(1+\frac{\gamma\mu}{\alpha})$$
then $\mu>\beta-\frac{1}{\beta}(\beta-\mu)^2.$ From this $$\frac{1}{\beta}(\beta-\mu)^2>\beta-\mu \ \Rightarrow \ \mu<0.$$. It is a contradiction for (\ref{parametr}).
Therefore, the fixed point ($x^*,y^*)$ is not attracting.

The inequality $|\lambda_{1,2}|>1$ is equivalent to $\Lambda_1<0$ \ or \ $\Lambda_2>2$ \ or \ $\Lambda_2<0,\ \Lambda_1>2$.

For the values of parameters given in (\ref{parametr}) the inequalities
$$\Lambda_1=\frac{1}{2}\left(\mu+A+\sqrt{(\mu-A)^2+4AB}\right)<0,\ \ \Lambda_2=\frac{1}{2}\left(\mu+A-\sqrt{(\mu-A)^2+4AB}\right)>2$$ do not hold. Next we check the case $\Lambda_2<0,\ \Lambda_1>2:$

\begin{equation}\label{l1l2}\left\{%
\begin{array}{ll}
    \Lambda_2=\frac{1}{2}\left(\mu+A-\sqrt{(\mu-A)^2+4AB}\right)<0  \\[2mm]
    \Lambda_1=\frac{1}{2}\left(\mu+A+\sqrt{(\mu-A)^2+4AB}\right)>2.
\end{array}%
\right.
\end{equation}
From the first inequality in the system (\ref{l1l2}) one has $\mu>0,$ i.e., the first inequality is always true under condition (\ref{parametr}). The second inequality of the system derives:
\begin{equation}
\alpha^2-2\left(\frac{\gamma\mu^2}{\beta-\mu}+\frac{\beta(2-\mu)}{2\beta+\mu(\beta-\mu)}\right)\alpha+\left(\frac{\gamma\mu^2}{\beta-\mu}\right)^2>0.
\end{equation}
Let
\begin{equation}\label{alpha12}
\begin{split}
\alpha_1&=\frac{\gamma\mu^2}{\beta-\mu}+\frac{\beta(2-\mu)}{2\beta+\mu(\beta-\mu)}+\sqrt{\frac{\beta(2-\mu)}{2\beta+\mu(\beta-\mu)}\left(\frac{2\gamma\mu^2}{\beta-\mu}+\frac{\beta(2-\mu)}{2\beta+\mu(\beta-\mu)}\right)},\\
 \alpha_2&=\frac{\gamma\mu^2}{\beta-\mu}+\frac{\beta(2-\mu)}{2\beta+\mu(\beta-\mu)}-\sqrt{\frac{\beta(2-\mu)}{2\beta+\mu(\beta-\mu)}\left(\frac{2\gamma\mu^2}{\beta-\mu}+\frac{\beta(2-\mu)}{2\beta+\mu(\beta-\mu)}\right)}.
 \end{split}
\end{equation}

It is obvious that the second inequality in (\ref{l1l2}) is not true under condition $\frac{\gamma\mu^2}{\beta-\mu}<\alpha<\alpha_1\leq1$ . Hence, the fixed point $(x^*,y^*)$ is not repelling. If  $\alpha_1<\alpha$, then the fixed point $(x^*,y^*)$ is repelling.

Thus for the type of fixed points the following proposition holds.

\begin{pro}\label{type} Let $\alpha_1, \ \alpha_2$ are defined in (\ref{alpha12}). The type of the fixed points for (\ref{systemacase}) are as follows:
\begin{itemize}
  \item[i)] if $\beta\leq\mu(1+\frac{\gamma\mu}{\alpha})$,  then the operator (\ref{systemacase}) has unique fixed point $(0,0)$ which is attracting.
  \item[ii)]if \ $\beta>\mu(1+\frac{\gamma\mu}{\alpha})$, then the operator has two fixed points $(0,0)$, $(x^*,y^*)$, and the point $(0,0)$ is attracting,  the point
  $$(x^*,y^*)=\left\{\begin{array}{lll}
  repelling, \ \ \mbox{if} \ \ \alpha >\alpha_1 \\[2mm]
  saddle, \ \ \mbox{if} \ \ \alpha<\alpha_1\leq1 \\[2mm]
  non-hyperbolic, \ \ \mbox{if} \ \  \alpha=\alpha_{1},\alpha_{2}.
  \end{array}\right.$$
  \end{itemize}
\end{pro}

\subsection{The limits of trajectories}\

The following theorem  describes the trajectory of any initial point $(x^{(0)}, y^{(0)})$ in $\mathbb{R}^2_{+}$.
\begin{thm}\label{pr} For the operator $W_{0}$ given by (\ref{systemacase}) for any initial point $(x^{(0)}, y^{(0)})\in \mathbb R^2_+$ the following hold:
\begin{itemize}
\item[(i)] If $y^{(n)}>\frac{\alpha}{\mu}$ for any natural number $n$ then $\lim\limits_{n\to \infty}x^{(n)}=+\infty, \ \lim\limits_{n\to \infty}y^{(n)}=\frac{\alpha}{\mu};$
\item[(ii)] If there exists $n_{0}$ number such that  $y^{(n_{0})}\leq\frac{\alpha}{\mu}$ and $\beta\leq\mu(1+\frac{\gamma\mu}{\alpha})$ then $\lim\limits_{n\to \infty}x^{(n)}=0, \ \lim\limits_{n\to \infty}y^{(n)}=0;$
\end{itemize}
where $(x^{(n)}, y^{(n)})=W_0^n(x^{(0)}, y^{(0)})$, with $W_{0}^n$ is $n$-th iteration of $W_{0}$.
\end{thm}
\begin{proof} First, we prove the assertion $(i)$. Let all values of $y^{(n)}$ are greater than $\frac{\alpha}{\mu}$. Then
$$y^{(n+1)}-y^{(n)}=\frac{\alpha x^{(n)}}{1+x^{(n)}}-\mu y^{(n)}<\alpha-\mu y^{(n)}=\mu(\frac{\alpha}{\mu}-y^{(n)})<0.$$ So $y^{(n)}$ is a decreasing sequence. Since  $y^{(n)}$ is decreasing and bounded from below we have:
\begin{equation}\label{y[n]>a}
\lim_{n\to \infty}y^{(n)}\geq\frac{\alpha}{\mu}.
\end{equation}
We estimate $y^{(n)}$ by the following:
 $$y^{(n)}=\frac{\alpha x^{(n-1)}}{1+x^{(n-1)}}+(1-\mu)y^{(n-1)}<\alpha+(1-\mu)y^{(n-1)}<\alpha+(1-\mu)(\alpha+(1-\mu)y^{(n-2)})$$
 $$<\alpha+\alpha(1-\mu)+(1-\mu)^2(\alpha+(1-\mu)y^{(n-3)})< ...<\alpha+\alpha(1-\mu)+\alpha(1-\mu)^2+...+\alpha(1-\mu)^{n-1}$$ $$+(1-\mu)^{n}y^{(0)}=\frac{\alpha}{\mu}+(1-\mu)^{n}(y^{(0)}-\frac{\alpha}{\mu}).$$
Thus $y^{(n)}<\frac{\alpha}{\mu}+(1-\mu)^{n}(y^{(0)}-\frac{\alpha}{\mu})$. Consequently
\begin{equation}\label{y[n]<a}
\lim_{n\to \infty}y^{(n)}\leq\frac{\alpha}{\mu}.
\end{equation}
 By (\ref{y[n]>a}) and (\ref{y[n]<a}) we have \begin{equation}\label{y[n]=a}
 \lim\limits_{n\to \infty}y^{(n)}=\frac{\alpha}{\mu}.
\end{equation}
From (\ref{y[n]=a}) and $y^{(n)}=\frac{\alpha x^{(n-1)}}{1+x^{(n-1)}}-\mu y^{(n-1)}+y^{(n-1)}$ it follows $\lim\limits_{n\to \infty}x^{(n)}=+\infty.$

Let's prove the assertion $(ii)$. If $y^{(n-1)}\leq\frac{\alpha}{\mu}$ then $y^{(n)}<\frac{\alpha}{\mu}$. Indeed,
  $$y^{(n)}=(1-\mu)y^{(n-1)}+\frac{\alpha x^{(n-1)}}{1+x^{(n-1)}}\leq(1-\mu)\frac{\alpha}{\mu}+\frac{\alpha x^{(n-1)}}{1+x^{(n-1)}}$$
  $$=\frac{\alpha}{\mu}-\alpha(1-\frac{x^{(n-1)}}{1+x^{(n-1)}})<\frac{\alpha}{\mu}.$$

Let
\begin{equation}\label{beta}
\beta<\mu(1+\frac{\gamma\mu}{\alpha}).
\end{equation}
If $\beta<\mu(1+\frac{\gamma\mu}{\alpha})$ then there exists $k>1$ such that $\beta\cdot k=\mu(1+\frac{\gamma\mu}{\alpha})$.
Denote
$c^{(n)}=x^{(n)}+y^{(n)}$ and $c^{(n)}_{0}=k\cdot x^{(n)}+y^{(n)}$, where $x^{(n)}, y^{(n)}$ defined by the following
\begin{equation}\label{recc}\begin{array}{ll}
x^{(n)}=\frac{\beta (y^{(n-1)})^2}{\gamma+y^{(n-1)}}-\frac{\alpha x^{(n-1)}}{1+x^{(n-1)}}+x^{(n-1)},\\[3mm]
 y^{(n)}=\frac{\alpha x^{(n-1)}}{1+x^{(n-1)}}-\mu y^{(n-1)}+y^{(n-1)},\ \ n=1,2,3,...\ .
 \end{array}\end{equation}

By $y^{(n-1)}\leq\frac{\alpha}{\mu}$ and $\beta<\mu(1+\frac{\gamma\mu}{\alpha})$ we have
\begin{equation}\label{x+y}
(\beta-\mu)y^{(n-1)}-\gamma\mu<0,\ \ (1+\frac{\gamma\mu}{\alpha})\frac{y^{(n-1)}}{\gamma+y^{(n-1)}}-1<0.
\end{equation}
By using (\ref{x+y}) to the equalities $c^{(n)}$ and $c_{0}^{(n)}$ we obtain the followings.

$$c^{(n)}=x^{(n)}+y^{(n)}=x^{(n-1)}+y^{(n-1)}+$$
$$+\frac{y^{(n-1)}}{\gamma+y^{(n-1)}}((\beta-\mu)y^{(n-1)}-\gamma\mu)< x^{(n-1)}+y^{(n-1)}=c^{(n-1)},$$
$$c_{0}^{(n)}=k\cdot x^{(n)}+y^{(n)}=k\cdot x^{(n-1)}+y^{(n-1)}
+(1-k)\frac{\alpha x^{(n-1)}}{1+x^{(n-1)}}+$$ $$ +\mu y^{(n-1)}((1+\frac{\gamma\mu}{\alpha})\frac{y^{(n-1)}}{\gamma+y^{(n-1)}}-1)<k\cdot x^{(n-1)}+y^{(n-1)}=c_{0}^{(n-1)}.$$

Hence both sequences $\{c^{(n)}\}$ and $\{c^{(n)}_{0}\}$ are monotone and bounded,
i.e.,
$$0<...< c^{(n)}< c^{(n-1)}<...< c^{(0)},$$
$$ 0<...< c^{(n)}_{0}< c^{(n-1)}_{0}<...< c^{(0)}_{0}.$$
Thus $\{c^{(n)}\}$ and $\{c^{(n)}_{0}\}$ have limit points, denote the limits by $c^*$ and $c^*_{0}$ respectively.
Consequently, the following limits exist
$$\tilde{x}=\lim_{n\to \infty}x^{(n)}=\frac{1}{1-k}\lim_{n\to \infty}(c^{(n)}-c^{(n)}_{0})=\frac{1}{1-k}(c^*-c^*_{0}),$$
$$\tilde{y}=\lim_{n\to \infty}y^{(n)}=c^*-\tilde{x}.$$
 and by (\ref{recc}) we have
$$\tilde{x}=\frac{\beta \tilde{y}^2}{\gamma+\tilde{y}}-\frac{\alpha \tilde{x}}{1+\tilde{x}}+\tilde{x},\ \
 \tilde{y}=\frac{\alpha \tilde{x}}{1+\tilde{x}}-\mu \tilde{y}+\tilde{y},$$
i.e., $\tilde{x}=0, \tilde{y}=0.$

In the case $\beta=\mu(1+\frac{\gamma\mu}{\alpha})$ also the monotone decreasing sequence $c^{(n)}$ is bounded from below. From the existence of the limit of $c^{(n)}$ and by (\ref{recc}) we have $\lim\limits_{n\to \infty}x^{(n)}=0, \ \lim\limits_{n\to \infty}y^{(n)}=0.$
\end{proof}

\subsection{Dynamics on invariant sets.}\

A set $A$ is called invariant with respect to $W_{0}$ if $W_{0}(A)\subset A$.

Denote
$$\Omega_1=\{(x,y)\in\mathbb{R}_+^{2}, \ \ 0\leq x\leq x^*, \ \ 0\leq y\leq y^*\}\setminus\{(x^*,y^*)\}$$
$$\Omega_2=\{(x,y)\in\mathbb{R}_+^{2}, \ \ x^*\leq x, y^*\leq y\}\setminus\{(x^*,y^*)\}$$
where $(x^*,y^*)$ is the fixed point defined by $(\ref{x^*y^*})$.

Let us consider the dynamics of the operator $W_{0}$ given by (\ref{systemacase}) in the sets $\Omega_1,\ \Omega_2$ under condition $\beta>\mu(1+\frac{\gamma\mu}{\alpha}).$
\begin{lemma} The sets $\Omega_1$ and $\Omega_2$ are invariant with respect to $W_{0}.$
\end{lemma}
\begin{proof}
\textbf{1)} Let $0\leq x\leq x^*, 0\leq y\leq y^*.$ Then
$$x'-x^*=\frac{\beta y^2}{\gamma+y}+x(1-\frac{\alpha}{1+x})-x^*\leq\frac{\beta y^{*2}}{\gamma+y^*}+x^*(1-\frac{\alpha}{1+x})-x^*$$
$$=\frac{\beta y^{*2}}{\gamma+\frac{\gamma\mu}{\beta-\mu}}-\frac{\alpha x^*}{1+x}\leq\frac{\beta-\mu}{\gamma}y^{*2}-\frac{\alpha x^*}{1+x^*}$$
$$=\frac{\beta-\mu}{\gamma}y^{*2}-\mu y^*=y^*(\frac{\beta-\mu}{\gamma}\cdot\frac{\gamma\mu}{\beta-\mu}-\mu)=0.$$
$y^*-y'=y^*-(1-\mu)y-\frac{\alpha x}{1+x}\geq y^*-(1-\mu)y^*-\frac{\alpha x^*}{1+x^*}=\mu y^*-\frac{\alpha x^*}{1+x^*}=0.$\\
Thus $(x',y')\in W_0(\Omega_1)\subset\Omega_1$

\textbf{2)} Let $x\geq x^*, y\geq y^*.$ Then
$$x^*-x'=x^*-\frac{\beta y^2}{\gamma+y}-x(1-\frac{\alpha}{1+x})\leq x^*-\frac{\beta y*^2}{\gamma+y^*}-x^*(1-\frac{\alpha}{1+x})$$
$$=\frac{\alpha x^*}{1+x}-\frac{\beta-\mu}{\gamma}y^{*^2}\leq\frac{\alpha x^*}{1+x^*}-\frac{\beta-\mu}{\gamma}y^{*^2}$$
$$=y^*(\mu-\frac{\beta-\mu}{\gamma}\frac{\gamma\mu}{\beta-\mu})=0.$$
$$y'-y^*=\frac{\alpha x}{1+x}+(1-\mu)y-y^*\geq\frac{\alpha x^*}{1+x^*}+(1-\mu)y^*-y^*=0.$$
Thus $(x',y')\in W_0(\Omega_2)\subset\Omega_2.$
\end{proof}

The following theorem  describes the trajectory of any point $(x^{(0)}, y^{(0)})$ in invariant sets.
\begin{thm}\label{pr1} For the operator $W_{0}$ given by (\ref{systemacase}) (i.e. under condition (\ref{parametr})),
if $\beta>\mu(1+\frac{\gamma\mu}{\alpha})$ then for any initial point $(x^{(0)}, y^{(0)})$, the following hold
$$\lim_{n\to \infty}x^{(n)}=\left\{\begin{array}{ll}
0, \ \ \ \ \ \ \mbox{if} \ \ (x^{(0)}, y^{(0)})\in\Omega_1, \\[2mm]
+\infty, \ \ \mbox{if} \ \ (x^{(0)}, y^{(0)})\in\Omega_2
\end{array}\right.$$
$$\lim_{n\to \infty}y^{(n)}=\left\{\begin{array}{ll}
0, \ \ \ \ \ \ \mbox{if} \ \ (x^{(0)}, y^{(0)})\in\Omega_1, \\[2mm]
\frac{\alpha}{\mu}, \ \ \ \ \  \mbox{if} \ \ (x^{(0)}, y^{(0)})\in\Omega_2
\end{array}\right.$$
where $(x^{(n)}, y^{(n)})=W_0^n(x^{(0)}, y^{(0)})$.
\end{thm}
\begin{proof} Adding $x^{(n)}$ and $y^{(n)}$ we get (see (\ref{x^*y^*}))
\begin{equation}\label{xn+yn}
x^{(n)}+y^{(n)}=x^{(n-1)}+y^{(n-1)}-\frac{(\beta-\mu)y^{(n-1)}}{\gamma+y^{(n-1)}}(y^*-y^{(n-1)}).
\end{equation}

We need to the following lemmas.
\begin{lemma}\label{border} For any parameters satisfying (\ref{parametr}) and for arbitrary
initial point $(x^{(0)}, y^{(0)})$ in $\mathbb{R}^2_{+}$ the sequence $y^{(n)}$ (defined in (\ref{recc})) is bounded:
$$0\leq y^{(n)}\leq \left\{\begin{array}{ll}
y^{(0)}, \ \ \mbox{if} \ \ y^{(0)}>\frac{\alpha}{\mu}\\[2mm]
\frac{\alpha}{\mu}, \ \ \mbox{if} \ \ y^{(0)}<\frac{\alpha}{\mu}.
\end{array}\right.$$
\end{lemma}
\begin{proof} Note that on the set $\mathbb{R}_{+}$ the function $y(x)=\frac{x}{1+x}$ is increasing and
 bounded by 1. Therefore
 $$y^{(n)}=\frac{\alpha x^{(n-1)}}{1+x^{(n-1)}}+(1-\mu)y^{(n-1)}\leq\alpha+(1-\mu)y^{(n-1)}\leq\alpha+(1-\mu)(\alpha+(1-\mu)y^{(n-2)})\leq$$
 $$\alpha+\alpha(1-\mu)+(1-\mu)^2(\alpha+(1-\mu)y^{(n-3)})\leq
 ...\leq\alpha+\alpha(1-\mu)+\alpha(1-\mu)^2+...+\alpha(1-\mu)^{n-1}$$ $$+(1-\mu)^{n}y^{(0)}=\frac{\alpha}{\mu}+(1-\mu)^{n}(y^{(0)}-\frac{\alpha}{\mu}).$$

 Thus if the initial point $y^{(0)}>\frac{\alpha}{\mu}$ then for any $m\in\mathbb{N}$ we have $0\leq y^{(m)}\leq y^{(0)}$. Moreover, if $y^{(0)}<\frac{\alpha}{\mu}$ then for any $m\in\mathbb{N}$ we have $0\leq y^{(m)}\leq \frac{\alpha}{\mu}.$
\end{proof}

\begin{lemma}\label{Theta1} For sequences $x^{(n)}$ and $y^{(n)}$ in the set $\Omega_1$ the following statements hold:
\begin{itemize}
  \item[1)] For any $n\in\mathbb{N}$ the inequalities
   $x^{(n)}<x^{(n+1)}$ and $y^{(n)}<y^{(n+1)}$ can not be satisfied at the same time.
  \item[2)] If  $x^{(m-1)}>x^{(m)}$, $y^{(m-1)}>y^{(m)}$ for some $m\in\mathbb{N}$ then $x^{(m)}>x^{(m+1)}$, $y^{(m)}>y^{(m+1)}$.
  \item[3)] If  $\beta>\mu(1+\frac{\gamma\mu}{\alpha})$ then $x^{(m-1)}>x^{(m)}$, $y^{(m-1)}<y^{(m)}$ can not be satisfied for any $m\in\mathbb{N}.$
  \item[4)] If $\beta>\mu(1+\frac{\gamma\mu}{\alpha})$ then $x^{(m-1)}<x^{(m)}$, $y^{(m-1)}>y^{(m)}$ can not be satisfied for any $m\in\mathbb{N}.$
\end{itemize}
\end{lemma}
\begin{proof} Let $(x^{(0)}, y^{(0)})\in\Omega_1.$ For each $n\in\mathbb{N}$, $x^{(n)}\leq x^*$ and $y^{(n)}\leq y^*.$
\begin{itemize}
  \item[1)] From (\ref{xn+yn})  we get $(x^{(n)}-x^{(n-1)})+(y^{(n)}-y^{(n-1)})<0.$
  Consequently, $x^{(n)}<x^{(n+1)}$ and $y^{(n)}<y^{(n+1)}$ can not be satisfied at the same time.
  \item[2)] Since the functions $u(x)=x(1-\frac{\alpha}{1+x}), v(y)=\frac{y^2}{\gamma+y}$ are monotonically increasing and by $x^{(m-1)}-x^{(m)}>0$,\ $y^{(m-1)}-y^{(m)}>0$  we have
  $$x^{(m-1)}(1-\frac{\alpha}{1+x^{(m-1)}})-x^{(m)}(1-\frac{\alpha}{1+x^{(m)}})>0, \ \frac{x^{(m-1)}}{1+x^{(m-1)}}>\frac{x^{(m)}}{1+x^{(m)}}$$ and
  $$\frac{y^{(m-1)^2}}{\gamma+y^{(m-1)}}>\frac{y^{(m)^2}}{\gamma+y^{(m)}}.$$
 Then
 $$x^{(m)}-x^{(m+1)}=\beta(\frac{y^{(m-1)^2}}{\gamma+y^{(m-1)}}-\frac{y^{(m)^2}}{\gamma+y^{(m)}})+x^{(m-1)}(1-\frac{\alpha}{1+x^{(m-1)}})-x^{(m)}(1-\frac{\alpha}{1+x^{(m)}})>0,$$ $$y^{(m)}-y^{(m+1)}=\alpha(\frac{x^{(m-1)}}{1+x^{(m-1)}}-\frac{x^{(m)}}{1+x^{(m)}})+(1-\mu)(y^{(m-1)}-y^{(m)})>0.$$
\item[3)] Let $\beta>\mu(1+\frac{\gamma\mu}{\alpha}).$ Assume   $x^{(m-1)}>x^{(m)}$, $y^{(m-1)}<y^{(m)}$ hold for any $m\in\mathbb{N}.$
      Then since $x^{(n)}$ is decreasing and bounded; $y^{(n)}$ is increasing and bounded (see Lemma \ref{border}) there exist their limits $\tilde{x}$, $\tilde{y}\neq0$ respectively. By (\ref{recc}) we obtain
$$\left\{
  \begin{array}{ll}
   \frac{\beta \tilde{y}^2}{\gamma+\tilde{y}} =\frac{\alpha \tilde{x}}{1+\tilde{x}}\\[2mm]
   \frac{\alpha \tilde{x}}{1+\tilde{x}}=\mu \tilde{y}\\
  \end{array}
\right.$$ i.e. $\tilde{x}=0, \tilde{y}=0.$ This contradiction shows that if  $\beta>\mu(1+\frac{\gamma\mu}{\alpha})$ then $x^{(m-1)}>x^{(m)}$, $y^{(m-1)}<y^{(m)}$ can not be satisfied for any $m\in\mathbb{N}.$
\item[4)] Let $\beta>\mu(1+\frac{\gamma\mu}{\alpha}).$ Assume $x^{(m-1)}<x^{(m)}$, $y^{(m-1)}>y^{(m)}$ hold for any $m\in\mathbb{N}.$  $(x^{(m)},y^{(m)})\in\Omega_1.$ Then since $x^{(n)}$ is increasing and bounded; $y^{(n)}$ is decreasing and bounded there exist their limits $\tilde{x}\neq0$, $\tilde{y}$ respectively. But by (\ref{recc}) we obtain
$\tilde{x}=\lim\limits_{m\rightarrow\infty}x^{(m)}=0.$ This completes proof of part 4.
\end{itemize}
\end{proof}
\begin{lemma}\label{Theta2} For sequences $x^{(n)}$ and $y^{(n)}$ in the set $\Omega_2$ the following statements hold:
\begin{itemize}
  \item[1)] For any $n\in\mathbb{N}$ the inequalities
   $x^{(n)}>x^{(n+1)}$ and $y^{(n)}>y^{(n+1)}$ can not be satisfied at the same time.
  \item[2)] If  $x^{(m-1)}<x^{(m)}$, $y^{(m-1)}<y^{(m)}$ for some $m\in\mathbb{N}$ then $x^{(m)}<x^{(m+1)}$, $y^{(m)}<y^{(m+1)}$.
  \item[3)] If  $\beta>\mu(1+\frac{\gamma\mu}{\alpha})$ then $x^{(m-1)}>x^{(m)}$, $y^{(m-1)}<y^{(m)}$ can not be satisfied for any $m\in\mathbb{N}.$
  \item[4)] If $\beta>\mu(1+\frac{\gamma\mu}{\alpha})$ then $x^{(m-1)}<x^{(m)}$, $y^{(m-1)}>y^{(m)}$ can not be satisfied for any $m\in\mathbb{N}.$
\end{itemize}
\end{lemma}
\begin{proof}
Let $(x^{(0)}, y^{(0)})\in\Omega_2.$ For each $n\in\mathbb{N}$, $x^{(n)}\geq x^*$ and $y^{(n)}\geq y^*.$
\begin{itemize}
  \item[1)] From (\ref{xn+yn})  we get $(x^{(n)}-x^{(n-1)})+(y^{(n)}-y^{(n-1)})>0.$
  Consequently, $x^{(n)}>x^{(n+1)}$ and $y^{(n)}>y^{(n+1)}$ can not be satisfied at the same time.
  \item[2)] By $x^{(m)}-x^{(m-1)}>0$, $y^{(m)}-y^{(m-1)}>0$ we have
  $$x^{(m)}(1-\frac{\alpha}{1+x^{(m)}})-x^{(m-1)}(1-\frac{\alpha}{1+x^{(m-1)}})>0, \ \frac{x^{(m)}}{1+x^{(m)}}>\frac{x^{(m-1)}}{1+x^{(m-1)}}$$ and
  $$\frac{y^{(m)^2}}{\gamma+y^{(m)}}>\frac{y^{(m-1)^2}}{\gamma+y^{(m-1)}}.$$
 Then
 $$x^{(m+1)}-x^{(m)}=\frac{\beta y^{(m)^2}}{\gamma+y^{(m)}}-\frac{\beta y^{(m-1)^2}}{\gamma+y^{(m-1)}}+x^{(m)}(1-\frac{\alpha}{1+x^{(m)}})-x^{(m-1)}(1-\frac{\alpha}{1+x^{(m-1)}})>0,$$ $$y^{(m+1)}-y^{(m)}=\alpha(\frac{x^{(m)}}{1+x^{(m)}}-\frac{x^{(m-1)}}{1+x^{(m-1)}})+(1-\mu)(y^{(m)}-y^{(m-1)})>0.$$
\item[3)] Similarly to the proof of part $3$ of Lemma \ref{Theta1}.
\item[4)] Let $\beta>\mu(1+\frac{\gamma\mu}{\alpha})$. Assume if for any $m\in\mathbb{N}$ the inequalities  $x^{(m-1)}<x^{(m)}$, $y^{(m-1)}>y^{(m)}$  are satisfied at the same time, i.e. $x^{(m)}$ is increasing and $y^{(m)}$ is decreasing.
       Let
       $$\Delta^{(m)}=(x^{(m+1)}-x^{(m)})+(y^{(m+1)}-y^{(m)})=\frac{(\beta-\mu)y^{(m)}}{\gamma+y^{(m)}}(y^{(m)}-y^*).$$

Since $\{y^{(m)}\}$ is decreasing, $\Delta^{(m)}>0$ for $\beta>\mu(1+\frac{\gamma\mu}{\alpha})$ we conclude that the sequence   $\{\Delta^{(m)}\}$ is decreasing and bounded from below. Thus $\{\Delta^{(m)}\}$ has a limit and since $y^{(m)}$ has limit we conclude that $x^{(m)}$ has a finite limit. By (\ref{recc}) we have
$$\lim_{m\rightarrow\infty}x^{(m)}=0, \ \ \lim_{m\rightarrow\infty}y^{(m)}=0.$$
But this is a contradiction to $\lim\limits_{m\rightarrow\infty}x^{(m)}\neq0$. This completes proof of part 4.
\end{itemize}
\end{proof}
\begin{lemma}\label{Theta12} If $x^{(n)}$ and $y^{(n)}$ in $\Omega_1,$ then there exists $n_0\in\mathbb{N}$ such that $x^{(n)}$ and $y^{(n)}$ are decreasing for $n\geq n_0$, if $x^{(n)}$ and $y^{(n)}$ in $\Omega_2,$ then there exists $m_0\in\mathbb{N}$ such that $x^{(n)}$ and $y^{(n)}$ are increasing for $n\geq m_0$.
\end{lemma}
\begin{proof} Monotonicity of $x^{(n)}$ and $y^{(n)}$  follow from Lemma \ref{Theta1} and Lemma \ref{Theta2}.
\end{proof}
\begin{lemma} If $x^{(n)}$ in the set $\Omega_2$ then $x^{(n)}$ is unbounded from above.
\end{lemma}
\begin{proof} There exists $n_0$ such that the sequences $x^{(n)}$ is increasing for $n\geq n_0.$
Consider
\begin{equation}\label{xsys}
\left\{
  \begin{array}{ll}
   (x^{(n_{0}+1)}-x^{(n_{0})})+(y^{(n_{0}+1)}-y^{(n_{0})})=\frac{(\beta-\mu)y^{(n_{0})}}{\gamma+y^{(n_{0})}}(y^{(n_{0})}-y^*) \\[2mm]
   (x^{(n_{0}+2)}-x^{(n_{0}+1)})+(y^{(n_{0}+2)}-y^{(n_{0}+1)})=\frac{(\beta-\mu)y^{(n_{0}+1)}}{\gamma+y^{(n_{0}+1)}}(y^{(n_{0}+1)}-y^*) \\
     ... \\
   (x^{(n-1)}-x^{(n-2)})+(y^{(n-1)}-y^{(n-2)})=\frac{(\beta-\mu)y^{(n-2)}}{\gamma+y^{(n-2)}}(y^{(n-2)}-y^*)\\[2mm]
   (x^{(n)}-x^{(n-1)})+(y^{(n)}-y^{(n-1)})=\frac{(\beta-\mu)y^{(n-1)}}{\gamma+y^{(n-1)}}(y^{(n-1)}-y^*)
  \end{array}
\right.
\end{equation}
Adding equations of (\ref{xsys}) we get

$$(x^{(n)}-x^{(n_{0})})+(y^{(n)}-y^{(n_{0})})$$
\begin{equation}\label{xsyst}
=(\beta-\mu)\left(\frac{y^{(n_{0})}}{\gamma+y^{(n_{0})}}(y^{(n_{0})}-y^*)+...+\frac{y^{(n-1)}}{\gamma+y^{(n-1)}}(y^{(n-1)}-y^*)\right).
\end{equation}
Let $y^{(n)}$ (see Lemma \ref{border}) is  bounded from above by $\theta$.
By (\ref{xsyst}) we have $$x^{(n)}>x^{(n_{0})}+y^{(n_{0})}-\theta+(\beta-
\mu)(n-n_{0})(y^{(n_{0})}-y^*)\cdot\frac{y^{(n_{0})}}{\gamma+y^{(n_{0})}}.$$
For $\beta>\mu(1+\frac{\gamma\mu}{\alpha})$ from $$\lim_{n\rightarrow\infty}\left(x^{(n_{0})}+y^{(n_{0})}-\theta+(\beta-
\mu)(n-n_{0})(y^{(n_{0})}-y^*)\cdot\frac{y^{(n_{0})}}{\gamma+y^{(n_{0})}}\right)=+\infty$$ it follows that  $x^{(n)}$ is not bounded from above.
\end{proof}

Now we continue the proof of theorem.

If $(x^{(0)}, y^{(0)})\in\Omega_1$ then by  Lemma \ref{Theta12} there exist their limits $\tilde{x}$, $\tilde{y}$ respectively.
By (\ref{recc}) we have
$$\lim_{n\rightarrow\infty}x^{(n)}=0, \ \ \lim_{n\rightarrow\infty}y^{(n)}=0.$$
If $(x^{(0)}, y^{(0)})\in\Omega_2$ then for $\beta>\mu(1+\frac{\gamma\mu}{\alpha})$ the sequence $y^{(n)}$ has limit $\tilde{y}$  (see Lemma \ref{border}).
 Consequently, by (\ref{recc}) and $\lim\limits_{n\rightarrow\infty}x^{(n)}=+\infty $ we get $\tilde{y}=\frac{\alpha}{\mu}$.
 Theorem is proved.
 \end{proof}

\subsection{On the set $\mathbb{R}^2_+\setminus(\Omega_1\bigcup\Omega_2)$}\

In the following examples, we show trajectories of initial points from the set $\mathbb{R}^2_+\setminus(\Omega_1\bigcup\Omega_2)$.

Let us consider the operator with parameter values $\alpha=0.8, \beta=0.9, \gamma=2, \mu=0.4$ satisfying the condition $\beta>\mu(1+\frac{\gamma\mu}{\alpha})$.
 Then by (\ref{x^*y^*}), we get $x^*=4$, $y^* = 1.6$.
\begin{ex} If the initial point is $x^{(0)}=0.2, y^{(0)}=4$ then the trajectory of system (\ref{systemacase}) is shown in the Fig.\ref{Fig.1}, i.e.,
$$\lim_{n\to \infty}x^{(n)}=0, \ \lim_{n\to \infty}y^{(n)}=0.$$
If the initial point is $x^{(0)}=0.2, y^{(0)}=5$ then the trajectory of system (\ref{systemacase}) is shown in the Fig.\ref{Fig.1}. In this case the first coordinate of the trajectory goes to infinite and the second coordinate has limit point $2$, i.e.,
$$\lim_{n\to \infty}x^{(n)}=+\infty, \ \lim_{n\to \infty}y^{(n)}=\frac{\alpha}{\mu}=2.$$
\end{ex}

\begin{ex} If the initial point is $x^{(0)}=5.6$, $y^{(0)}=0.2$ then the trajectory of system (\ref{systemacase}) is shown in the Fig.\ref{Fig.2}, i.e., it converges to $(0,0)$.

If the initial point is $x^{(0)}=7$, $y^{(0)}=0.2$ then the trajectory of system (\ref{systemacase}) is shown in the Fig.\ref{Fig.2}. In this case the first coordinate of the trajectory goes to infinite and the second coordinate has limit point $2$.
\end{ex}
\begin{figure}[h!]
\begin{multicols}{2}
\hfill
\includegraphics[width=0.52\textwidth]{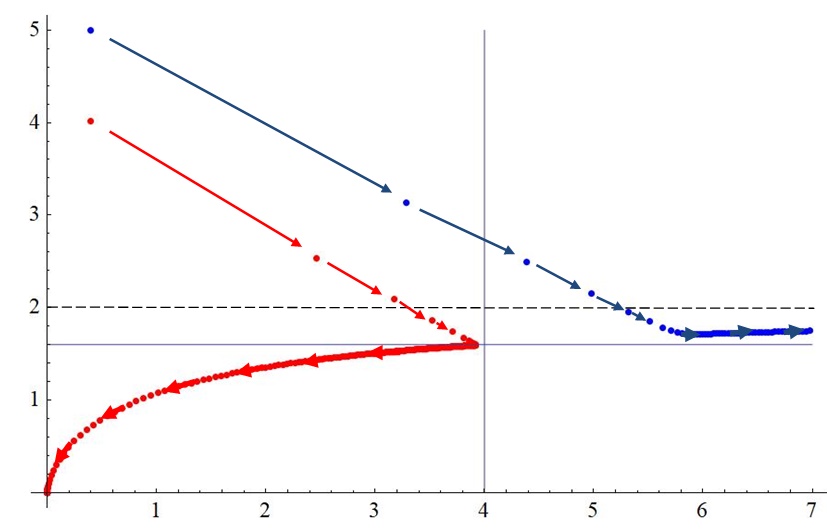}\\
\hfill
\caption{$x^{(0)}=0.2,\ y^{(0)}=4$ and $x^{(0)}=0.2,\ y^{(0)}=5$}\label{Fig.1}
\hfill
\includegraphics[width=0.52\textwidth]{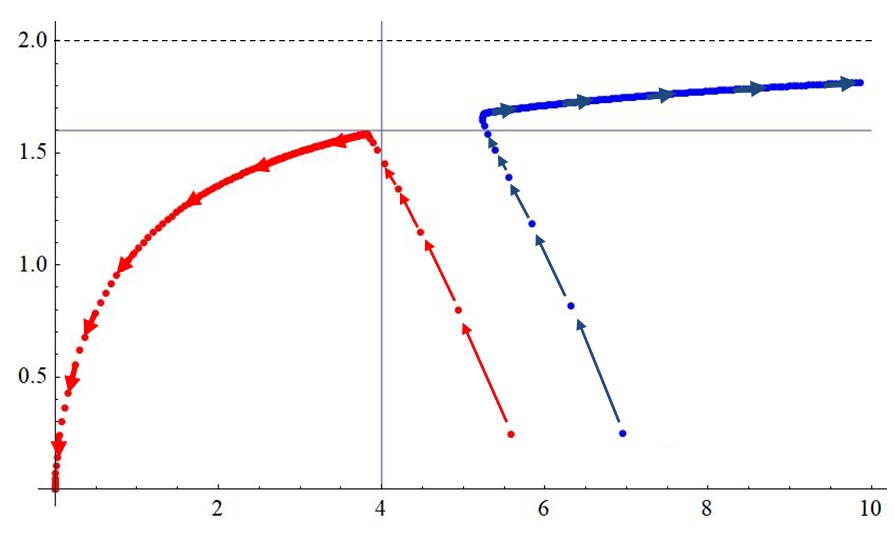}\\
\hfill
\caption{$x^{(0)}=5.6,\ y^{(0)}=0.2$ and $x^{(0)}=7,\ y^{(0)}=0.2$}\label{Fig.2}
\hfill
\end{multicols}
\end{figure}

\begin{rk}
We note that for continuous time system (2) the following results are known (see \cite{J.Li}):  All fixed (equilibrium) points are found and their types are determined. Moreover, local behavior of the dynamical system in the neighborhood   of the fixed point $(0,0)$ is studied. In our discrete-time case, for  $d_0=d_1=0$ we also determined types of all fixed points. Besides this we have been able to study global behavior of the system in neighborhood of $(0,0)$. Also we have found invariant sets and studied the dynamical system on the sets. The last results are not known for the continuous time.
\end{rk}

\section{Biological interpretations}
Each point (vector) $z=(x;y)\in \mathbb{R}^{2}_{+}$ can be considered as a state (a measure) of the mosquito
population.

Let us give some interpretations of our main results:
\begin{itemize}
\item[(a)] (Case Theorem \ref{pr}) By (\ref{beta}) we have $\gamma\geq\frac{\alpha(\beta-\mu)}{\mu^2}.$ Under this condition on $\gamma$ (i.e. on Allee effects), the mosquito population dies;
\item[(b)] (Case Theorem \ref{pr1}) If the inequality $\gamma<\frac{\alpha(\beta-\mu)}{\mu^2}$ holds for Allee effects $\gamma$, then extinction or survival of the mosquito population depends on their initial state.
 \end{itemize}

\end{document}